\documentclass[11pt]{amsart}
\usepackage[a4paper,margin=1.05in]{geometry}
\usepackage{amsmath,amssymb,amsthm}
\usepackage{microtype}
\usepackage[hidelinks]{hyperref}
\numberwithin{equation}{section} 

\newtheorem{theorem}{Theorem}[section]
\newtheorem{lemma}[theorem]{Lemma}
\newcommand{\R}{\mathbb R}
\newcommand{\tr}{\operatorname{tr}}

\newcommand{\eps}{\varepsilon}
\theoremstyle{remark}
\newtheorem{remark}[theorem]{Remark}

\title[The singular set of the superconductivity free boundary problem]
{The singular set of the superconductivity free boundary problem}

\author{Shibing Chen}
\address{}
\email{chenshibing1982@hotmail.com}

\author{Yibin Feng}
\address{Yibin Feng, School of Mathematics and Statistics,
	Lanzhou University, Lanzhou, 730000, P.R. China.}
\email{fengyb@lzu.edu.cn}

\author{Yuanyuan Li}
\address{Institute for Theoretical Sciences, Westlake University, Hangzhou 310030, China}
\email{lyyuan@westlake.edu.cn}

\author{Xianduo Wang}
\address{School of Mathematical Sciences, University of Science and Technology of China, Hefei, Anhui 230026, China}
\email{xdwang@ustc.edu.cn}

\date{}

\begin{document}

\begin{abstract}
We prove the uniqueness and continuous dependence of blow-ups at singular points of the free boundary for the superconductivity problem. As an application, we show that the free boundary is contained in a countable union of $C^1$ submanifolds of dimension at most $n-1$.
\end{abstract}

\maketitle

\section{Introduction}

In this paper, we study the singular set of the following free boundary problem
\begin{equation}\label{maineq}
  \Delta u=\chi_{\{|Du|>0\}}\qquad\text{in } D,
\end{equation}
where $D\subset \mathbb{R}^n$ is an open set. The equation is understood in the sense of distributions. It was proved in \cite[Theorem~2.14]{PSU2012} that $u\in C^{1,1}_{\mathrm{loc}}(D)$. We denote by $\Omega(u)=\{|Du|>0\}$ the active set, by $\Lambda(u)=\{Du=0\}$ the inactive set, and by $\Gamma(u)=\partial\Omega(u)\cap D$ the free boundary.

The problem  was studied by Caffarelli, Salazar and Shahgholian \cite{CSS2004} in connection with superconductivity. The existence of solutions to the associated Dirichlet problem and interior $C^{1,1}$ estimates were established by Caffarelli and Salazar \cite{CaffarelliSalazar2002}. The classification of blow-ups and $C^1$ regularity of the free boundary near regular points were obtained in \cite{CSS2004}.

We recall some results on the singular set for the classical obstacle problem
$$
\Delta u=\chi_{\{u>0\}},\qquad u\ge0.
$$
In a landmark work, Caffarelli \cite{Caffarelli1977,Caffarelli1998} proved that the free boundary can be decomposed into regular and singular set and that the singular set admits a $C^1$ stratification. At a singular point, the rescalings converge to a nonnegative quadratic polynomial. The singular set are divided into strata according to the dimension of the kernel of the blow-up. Each stratum is contained in a countable union of $C^1$ manifolds of the corresponding dimension. Weiss discovered a monotonicity formula which is used to study the homogeneity of blow-ups \cite{Weiss1999}. Based on this formula, Monneau obtained a monotonicity formula for the squared distance between a rescaling and a nonnegative quadratic polynomial \cite{Monneau2003,PSU2012}. By taking a subsequential blow-up as the comparison polynomial, one obtains uniqueness of blow-ups. The same formula also gives continuous dependence of the blow-up on the singular point. For higher regularity of the strata, we refer to the important work of Figalli and Serra \cite{FigalliSerra2019}.

We also recall the no-sign obstacle problem
$$
  \Delta u=1-\chi_{\{u=0,\ Du=0\}}.
$$
Caffarelli, Karp and Shahgholian studied the associated overdetermined problem in connection with the Pompeiu problem \cite{CKS2000}. Caffarelli and Shahgholian obtained the stratification of the singular set according to the kernels of the quadratic blow-ups \cite{CaffarelliShahgholian2004,PSU2012}. At a singular point, the matrixes of the blow-ups may be indefinite, and the uniqueness of blow-ups was left as an open problem \cite[Notes of Chapter 7, page 149]{PSU2012}. In \cite{ChenFengLi2022}, Chen, Feng and Li proved uniqueness of blow-ups by applying Monneau's formula to nonnegative quadratic comparison polynomials. The main observation is that the limiting distances to these comparison polynomials determine the blow-up matrix.
More recently, Jeon and Shahgholian \cite{JeonShahgholian2025} studied the finer structure of the singular set for the no-sign obstacle problem at points with nonnegative quadratic blow-ups, using the approach of Figalli and Serra.

For \eqref{maineq}, $u$ is constant on each connected component of the interior of $\Lambda(u)$, but these constants may be different. At a singular point $x$, we do not know the sign of $u-u(x)$ on on $\Lambda(u)$. In this problem, terms involving  $u-u(x)$ on the inactive set occur in both the Weiss and Monneau formulas. As pointed out in \cite[Notes to Chapter~7]{PSU2012}, the structure of the singular set for this problem remained largely open.  By using the idea of Chen, Feng and Li \cite{ChenFengLi2022}, in \cite{DuTangWang2024} the authors proved uniqueness under the assumption that $u(x)$ is a maximum of $u$ on the nearby inactive set.

Our proof is based on the comparison argument of Chen, Feng and Li \cite{ChenFengLi2022}. Their key observation is that the distances to nonnegative quadratic comparison polynomials determine the blow-up, even when the blow-up itself changes sign. For \eqref{maineq}, we construct a corrected projection
matrix that allows us to apply this argument. This matrix is the trace-one counterpart of the Hessian projection used by Andersson, Shahgholian and Weiss \cite{ASW2010}. We further use the comparison method to estimate the total variation of the projection, which gives continuous dependence of the blow-ups on the singular point.

For $x\in\Gamma(u)$ and $r>0$, we denote
$$
u_{x,r}(y)=\frac{u(x+ry)-u(x)}{r^2}.
$$
A blow-up at $x$ is a subsequential limit of $u_{x,r}$ as $r\downarrow0$ in $C^{1,\alpha}_{\rm loc}(\R^n)$, $0<\alpha<1$. By \cite{CSS2004}, at a free boundary point, the blow-ups are either half-space solutions $h_\nu(y)=\frac12((y\cdot\nu)^+)^2$, $|\nu|=1$, or homogeneous quadratic polynomials $q(y)=\frac12 y\cdot Ay$, with $A=A^T$ and $\tr A=1$. The point is called regular in the former case and singular in the latter case. Denote the singular set by $\Sigma(u)$. Our main results are the following.

\begin{theorem}\label{continuity}
Let $u$ solve \eqref{maineq}. The blow-up at each $x\in\Sigma(u)$ is unique and has the form
$$
q_x(y)=\frac12 y\cdot A_xy,
\qquad A_x=A_x^T,\quad \tr A_x=1.
$$
The matrix $A_x$ depends continuously on $x\in\Sigma(u)$.
\end{theorem}

For $d=0,\ldots,n-1$, set
$$
\Sigma_d(u)=\{x\in\Sigma(u):\dim\ker A_x=d\}.
$$

\begin{theorem}\label{manifolds}
Let $u$ solve \eqref{maineq}. For each $d=0,\ldots,n-1$, the set $\Sigma_d(u)$ is contained in a countable union of $d$-dimensional $C^1$ submanifolds of $\R^n$.
\end{theorem}
The above theorem implies that the entire free boundary is contained in a countable union of $C^1$ submanifolds of dimensions at most $n-1$.

We briefly explain the proof. We first consider the projection of $u_{x,r}$ onto the space of homogeneous harmonic quadratic polynomials. By differentiating this projection, we obtain an identity involving a volume moment of the inactive set. After subtracting a suitable multiple of this moment, we obtain a matrix whose derivative involves only the corresponding surface moment. Note that the rescaled gradient vanishes on the inactive part of the sphere. If a rescaling is close to a quadratic polynomial, the corrected matrix is small when applied to unit vectors in this part of the sphere. It follows that the distance from this matrix to any fixed positive definite matrix of trace one is monotone at small scales. By the argument of Chen, Feng and Li, we conclude that any two quadratic blow-ups are equal.

To prove continuous dependence, we choose a comparison matrix according to the eigenvalues of the corrected matrix at a fixed scale. We then estimate the total variation of the projection at all smaller scales by the error in the quadratic approximation. Under uniform local Hessian bounds, this error tends to zero uniformly at singular points. We obtain a uniform first-order expansion of $Du$ on the singular set. By the Whitney extension theorem and the implicit function theorem, we obtain the local manifolds in Theorem~\ref{manifolds}.

\section{Preliminaries}\label{prelim}

\subsection{Solutions and rescalings}
Let $u$, $u_{x,r}$, $D$, $\Omega(u)$, $\Gamma(u)$ and $\Sigma(u)$ be as in the introduction. For $x\in\Gamma(u)$, the rescaling $u_{x,r}$ solves \eqref{maineq} in $(D-x)/r$, with $u_{x,r}(0)=|Du_{x,r}(0)|=0$. By the local $C^{1,1}$ estimates, rescalings at free boundary points in a fixed compact subset of $D$ are precompact in $C^{1,\alpha}_{\mathrm{loc}}(\R^n)$ as $r\to0$, for every $0<\alpha<1$; see \cite[Section~3.3]{PSU2012}.

We also recall the following stability result. Suppose $u_j$ are solutions of \eqref{maineq} on a common open set and converge there in $C^{1,\alpha}_{\mathrm{loc}}$ for some $0<\alpha<1$. Then the limit solves \eqref{maineq}. Moreover, if $x_j\in\Gamma(u_j)$ and $x_j\to x$ in that open set, we have $x\in\Gamma(u)$ \cite[Proposition 3.17(iii)--(iv)]{PSU2012}. It follows that any limit of the above rescalings is a global solution, namely, the limit is a function $v:\R^n\rightarrow \mathbb{R}$  satisfies \eqref{maineq} and $\|D^2v\|_{L^\infty(\mathbb{R}^n)}\leq C$ and $0\in\Gamma(v)$. For such a limit, we have $v(0)=0$ and $Dv(0)=0$.

\subsection{Classification and a regularity criterion}

We consider the class
$$
 \mathcal Q=\{q:\ q\text{ is a homogeneous quadratic polynomial and }\Delta q=1\}.
$$
We recall the following results for $\mathcal Q$ and the half-space solutions $h_\nu$.
\begin{enumerate}
\item At a fixed free boundary point every blow-up belongs to $\mathcal Q$ or is a half-space solution $h_\nu$, and all the blow-ups at that point have the same type. In particular, a point is singular if and only if all its blow-ups belong to $\mathcal Q$, and regular if and only if all its blow-ups are half-space solutions \cite[Theorems 3.22--3.23]{PSU2012}, \cite[Lemma 2 and Proposition 1]{CSS2004}.
\item If $v$ is a global solution with bounded Hessian, $0\in\Gamma(v)$ and $v(0)=0$, then $v$ is a quadratic polynomial or it is convex with $v\ge0$ and
$\Lambda(v)=\{v=0\}$. In the convex case, the minimum value is zero since $0\in\Gamma(v)\subset\Lambda(v)$ \cite[Theorem 5.2]{PSU2012}.
\item Near a regular point of \eqref{maineq}, the free boundary is a $C^1$ graph \cite[Corollary 4.12]{PSU2012}.
\end{enumerate}

The following criterion follows from \cite[Lemma~4.13]{PSU2012} and the classification of blow-ups.
\begin{lemma}\label{halfspace}
There exists $\eps_*=\eps_*(n)>0$ such that if $w$ solves \eqref{maineq} in $B_1$, $0\in\Gamma(w)$, $\|D^2w\|_{L^\infty(B_1)}<\infty$, and $\|w-h_\nu\|_{C^1(B_1)}<\eps_*$ for some unit vector $\nu$, then $0$ is regular.
\end{lemma}

\begin{proof}
For any unit vector $e$ with $e\cdot\nu\ge 1/2$, we have $4\partial_e h_\nu-|Dh_\nu|^2\ge0$ in $B_1$. By the assumption, the corresponding expression for $w$ is bounded below by $-7\eps_*$ if $\eps_*\le1$. Choosing $7\eps_*<1/(4n^2)$, by the comparison principle in \cite[Lemma 4.13]{PSU2012} we obtain
$4\partial_e w-|Dw|^2\ge0$ in $B_{1/2}$. In particular, $\partial_e w\ge0$ near the origin. Suppose $q(y)=\tfrac12 y\cdot Ay$ is a quadratic blow-up at the origin. Then we have $(Ae)\cdot y\ge0$ for all $y\in\R^n$ and all such $e$,  it follows that $Ae=0$ for all such $e.$ Since all such $e$ spans $\mathbb{R}^n,$ it implies that $A=0,$ which contradicts $\tr A=1$.
\end{proof}

\subsection{Compactness at singular points}\label{polylimits}
The following compactness lemma follows from the proof of \cite[Lemmas~7.2--7.3]{PSU2012}, using \cite[Theorem~6.6]{PSU2012} for Problem B.

\begin{lemma}\label{cpt111}
Let $u$ solve \eqref{maineq} in $D$, and suppose $x_j\in\Sigma(u)$ stay in a compact subset of $D$. If $r_j>0$ and $r_j\to0$, then the rescalings $u_{x_j,r_j}$ are precompact in $C^{1,\alpha}_{\rm loc}(\R^n)$ for $0<\alpha<1$, and each limit is a homogeneous quadratic polynomial with Laplacian one.
\end{lemma}
\begin{proof}
By the local $C^{1,1}$ estimates, the sequence $w_j=u_{x_j,r_j}$ is precompact in $C^{1,\alpha}_{\mathrm{loc}}(\R^n)$. Let $v$ be any subsequential limit; then $D^2v$ is bounded. By \cite[Theorem~6.6]{PSU2012}, the inactive set is uniformly thin at singular points in compact subsets of $D$. Consequently, there exists a modulus of continuity $\sigma$, depending on $u$ and the fixed compact set containing the points $x_j$, but independent of $j$, such that
$$
 |\Lambda(w_j)\cap B_R|
 \le C_nR^n\sigma(Rr_j)\longrightarrow0
$$
for every fixed $R>0$. Passing to the limit in $\Delta w_j=1-\chi_{\Lambda(w_j)}$, we obtain $\Delta v=1$ in $\R^n$. Each second derivative of $v$ is bounded and harmonic, and hence constant by Liouville's theorem. Since $v(0)=0$ and $Dv(0)=0$, we conclude that $v\in\mathcal Q$.
\end{proof}

\section{Proofs of the main results}\label{comparisons}

\subsection{The equation for the projection}

For a fixed free boundary point $x$ and $r=e^{-t}$ with $\overline B_r(x)\subset D$, we denote $u_t=u_{x,r}$ and $\Lambda_t=\Lambda(u_t)$. Let $h_t$ be the orthogonal projection of $u_t-|y|^2/(2n)$ in $L^2(\partial B_1)$ onto the space of homogeneous harmonic quadratic polynomials. We set
$$
 q_t(y)=\frac{|y|^2}{2n}+h_t(y)=\tfrac12 y\cdot A(t)y.
$$
Then $A(t)=A(t)^T$, $\tr A(t)=1$, and $u_t-q_t$ is orthogonal in $L^2(\partial B_1)$ to the space of homogeneous harmonic quadratic polynomials. We also denote
$$
 \begin{aligned}
 M(t)&=\int_{B_1\cap\Lambda_t}y\otimes y\,dy,\qquad F(t)=\tr M(t),\\
 H(t)&=\int_{\partial B_1\cap\Lambda_t}
       \theta\otimes\theta\,dS_\theta.
 \end{aligned}
$$

The following differentiation identity for the spherical projection is a special case of \cite[Lemma~2.10]{IMN2017}. We include the computation in our notation.
For any symmetric matrix $C$ with $\tr C=0$, set $h(y)=\tfrac12 y\cdot Cy$. Since $\partial_tu_t=2u_t-y\cdot Du_t$, by Green's identity and \eqref{maineq} we have
\begin{equation}\label{moment}
 \begin{split}
 \frac{d}{dt}\int_{\partial B_1}u_t h\,dS
 &=\int_{\partial B_1}(2u_t-\partial_\nu u_t)h\,dS\\
 &=-\int_{B_1}h\Delta u_t\,dy
 =\int_{B_1\cap\Lambda_t}h\,dy
 =\frac12\tr(M(t)C).
 \end{split}
\end{equation}
Here we have used $\Delta h=0,$ $\partial_\nu h=2h$ on $\partial B_1$, and $\int_{B_1}h\,dy=0$. On the other hand, since $u_t-q_t$ is orthogonal to the
space of homogeneous harmonic quadratic polynomials in $L^2(\partial B_1),$  by a direct computation we have
\begin{equation}\label{pairing}
 \int_{\partial B_1}u_t h\,dS
 =\int_{\partial B_1}q_t h\,dS
 =\frac{|\partial B_1|}{2n(n+2)}\tr(A(t)C).
\end{equation}
By differentiating \eqref{pairing} and using \eqref{moment}, we have
$$
 \tr(A'(t)C)=\kappa_n\tr(M(t)C),
 \qquad \kappa_n=\frac{n(n+2)}{|\partial B_1|}.
$$
Since this holds for every symmetric trace-free matrix $C$, the matrix $A'(t)-\kappa_nM(t)$ is a scalar multiple of $I$. Taking the trace and using $\tr A'(t)=0$ and $\tr M(t)=F(t)$, we conclude that
\begin{equation}\label{projeq111}
 A'(t)=\kappa_n\left(M(t)-\frac{F(t)}n I\right).
\end{equation}

In the original variables,
$$ 
M(t)=r^{-n-2}\int_{B_r(x)\cap\Lambda(u)}(z-x)\otimes(z-x)\,dz.
$$
By writing the integral in polar coordinates, we see that it is locally absolutely continuous in $r$. Since $dr/dt=-r$, by a direct computation we have
\begin{equation}\label{mevolu111}
 \begin{split}
 M'(t)&=(n+2)M(t)-H(t),\\
 F'(t)&=(n+2)F(t)-\tr H(t)
 \end{split}
\end{equation}
for almost every $t$. 
 Define
\begin{equation}\label{corrected}
 \widehat A(t)=A(t)-\frac{\kappa_n}{n+2}
 \left(M(t)-\frac{F(t)}n I\right).
\end{equation}
Then $\widehat A$ is locally absolutely continuous and has trace one. By \eqref{projeq111} and \eqref{mevolu111}, we obtain, for almost every $t$,
\begin{equation}\label{surface}
 \begin{split}
 \widehat A'(t)
 &=A'(t)-\frac{\kappa_n}{n+2}
   \left(M'(t)-\frac{F'(t)}n I\right)\\
 &=\frac{\kappa_n}{n+2}
   \left(H(t)-\frac{\tr H(t)}n I\right).
 \end{split}
\end{equation}
The corrected matrix is related to the Hessian projection of Andersson, Shahgholian and Weiss \cite[Definition~4.2]{ASW2010}.

Let $\widehat q_t(y)=\tfrac12 y\cdot\widehat A(t)y$. For any fixed $\bar q(y)=\tfrac12 y\cdot By\in\mathcal Q$, define
\begin{equation}\label{compdist}
 D_B(t)=\tfrac12\|\widehat A(t)-B\|^2
 =\kappa_n\int_{\partial B_1}|\widehat q_t-\bar q|^2\,dS,
\end{equation}
where $\|A\|:=\sqrt{\tr A^TA}$ is the Frobenius norm.  Indeed, \eqref{compdist} follows from the same computation leading to \eqref{pairing}, since $\tr(\widehat A(t)-B)=0$. Differentiating \eqref{compdist} and using \eqref{surface}, we obtain, for almost every $t$,
\begin{equation}\label{comparison}
 \begin{split}
 D_B'(t)
 &=\tr\big((\widehat A(t)-B)\widehat A'(t)\big)\\
 &=\frac{\kappa_n}{n+2}
   \int_{\partial B_1\cap\Lambda_t}
   \theta\cdot(\widehat A(t)-B)\theta\,dS_\theta.
 \end{split}
\end{equation}

\subsection{Approximation at singular points}\label{approx}
We denote the projection error by
\begin{equation}\label{projerror}
 E(t)=\|Du_t-A(t)y\|_{L^\infty(B_1)}.
\end{equation}
When the base point varies, we write $A_z(t)$ and $E_z(t)$ for the projection matrix and error at $z$. For every compact $K\Subset D$, by Lemma~\ref{cpt111} we have
\begin{equation}\label{projbound}
 \lim_{T\to\infty}
 \sup_{\substack{z\in K\cap\Sigma(u)\\ t\ge T}}E_z(t)=0.
\end{equation}
Indeed, otherwise there are $z_j\in K\cap\Sigma(u)$, $t_j\to\infty$, and $\eps_0>0$ such that $E_{z_j}(t_j)\ge\eps_0$. By Lemma~\ref{cpt111}, after passing to a subsequence, $u_{z_j,e^{-t_j}}$ converges in $C^1(\overline B_1)$ to some $q\in\mathcal Q$. By continuity of the projection, the projections of $u_{z_j,e^{-t_j}}$ onto $\mathcal Q$ converge to $q$. Hence $A_{z_j}(t_j)\to D^2q$, and \eqref{projerror} implies $E_{z_j}(t_j)\to0$, contradicting $E_{z_j}(t_j)\ge\eps_0$.

Since $Du_t=0$ on $\Lambda_t$, by \eqref{projerror}, we have $|A(t)y|\le E(t)$ on $\Lambda_t\cap\overline B_1$. Since $\tr A(t)=1$, there is an eigenvalue $\lambda\ge1/n$ with a unit eigenvector $e$.  Hence,
$$ 
\frac1n|y\cdot e|\le\lambda|y\cdot e|=|A(t)y\cdot e|\le|A(t)y|\le E(t)\qquad(y\in\Lambda_t\cap\overline B_1).$$
This implies
$$ 
\Lambda_t\cap\overline B_1\subset\{y\in\overline B_1:|y\cdot e|\le nE(t)\},
$$
whose volume is at most $C_nE(t)$. Since $M(t)$ is positive semidefinite,
$$
 \|M(t)\|\le F(t)
 =\int_{B_1\cap\Lambda_t}|y|^2\,dy
 \le|\Lambda_t\cap B_1|\le C_nE(t).
$$
Combining these bounds with \eqref{corrected} and $|A(t)\theta|\le E(t)$ on $\partial B_1\cap\Lambda_t$, we obtain
\begin{equation}\label{sest123}
 F(t)+\|\widehat A(t)-A(t)\|\le C_nE(t),\qquad
 |\widehat A(t)\theta|\le C_nE(t)
 \quad(\theta\in\partial B_1\cap\Lambda_t).
\end{equation}

\subsection{Uniqueness}

Fix a positive definite matrix $B$ with $\tr B=1$. By \eqref{projbound} and \eqref{sest123}, for all sufficiently large $t$,
$$
 \theta\cdot(\widehat A(t)-B)\theta
 \le C_nE(t)-\lambda_{\min}(B)
 \le-\tfrac12\lambda_{\min}(B)
 \quad(\theta\in\partial B_1\cap\Lambda_t).
$$
It follows from \eqref{comparison} that $D_B$ is eventually nonincreasing. Since $D_B\ge0$, it has a limit as $t\to\infty$.

Suppose $q(y)=\tfrac12 y\cdot Ay$ and $\widetilde q(y)=\tfrac12 y\cdot\widetilde Ay$ are two blow-ups at $x$. Along the corresponding sequences of scales, the projection matrices converge to $A$ and $\widetilde A$, respectively. By \eqref{sest123}, $\|\widehat A(t)-A(t)\|\to0$. Since $D_B(t)$ has a limit as $t\to\infty$, passing to the limit in \eqref{compdist}, we obtain
\begin{equation}\label{blowdist}
 \|A-B\|^2=\|\widetilde A-B\|^2
 \qquad\text{for every }B>0\text{ with }\tr B=1.
\end{equation}
We now apply the argument of Chen, Feng and Li \cite[Section~2]{ChenFengLi2022}. For small $s>0$, the matrices $B_\pm=I/n\pm s(A-\widetilde A)$ are positive definite and have trace one. Take $B=B_\pm$ in \eqref{blowdist}, and a direct computation shows
$$ 
0=-2\tr\big((A-\widetilde A)(B_+-B_-)\big)
   =-4s\|A-\widetilde A\|^2.
$$
Hence $A=\widetilde A$. By compactness, $u_{x,r}\to q_x$ as $r\downarrow0$ in $C^{1,\alpha}_{\mathrm{loc}}(\R^n)$, $0<\alpha<1$, and $A(t)\to A_x$ as $t\to\infty$.

\subsection{Uniform convergence and continuity}\label{stability}
To obtain a quantitative estimate from the comparison argument of Chen, Feng and Li \cite{ChenFengLi2022}, we choose a comparison matrix adapted to the eigenvalues of $\widehat A(T)$.

\begin{lemma}\label{variest}
Let $\eps_T=\sup_{t\ge T}E(t)$. There exists a constant $\epsilon_n$ depending only on $n,$ if $\eps_T<\epsilon_n$, then
\begin{equation}\label{variation}
 \int_T^\infty
 \big(\tr H(t)+F(t)+\|A'(t)\|\big)\,dt\le C_n\eps_T.
\end{equation}
\end{lemma}

\begin{proof}
If $\eps_T=0$, then \eqref{sest123} and \eqref{projeq111} imply $M=F=A'=0$ for $t\ge T$, and \eqref{mevolu111} implies $H=0$ almost everywhere. Thus we may assume that $\eps_T>0$. Let $C_n$ be the constant in \eqref{sest123}.

Let $K$ be a large constant depending only on $n$, to be determined later. Among the $n+1$ disjoint intervals $[K^{2j+1}C_n\eps_T,K^{2j+3}C_n\eps_T)$, $0\le j\le n$, there is one containing no absolute value of an eigenvalue of $\widehat A(T)$. Suppose the chosen interval is $[\rho/K,  K\rho),$ where $\rho=K^{2j_0+2}C_n\eps_T.$

Let $\mathcal{S}$ be the subspace spanned by the eigenvectors of $\widehat A(T)$ whose corresponding eigenvalues satisfy $|\lambda|<\rho/K$. Let $P$ denote the orthogonal projection onto this subspace. By the choice of $\rho$, no eigenvalue of $\widehat A(T)$ has absolute value in $[\rho/K,K\rho)$. Thus
 \begin{equation}\label{low123}
 \text{every eigenvalue of}\ \widehat A(T)\
\text{on}\ \operatorname{im}(I-P)\ \text{satisfies}\ |\lambda|\ge K\rho.
\end{equation}
Taking $\eps_T$ sufficiently small so that $\rho/K<1/n$. Since $\tr \widehat A(T)=1$, it implies that $\mathcal{S}\ne \mathbb{R}^n,$ since otherwise 
$\tr \widehat A(T)<1.$ Hence $\tr P\le n-1$. Define
\begin{equation}\label{cmatrix12}
 B=\widehat A(T)+\frac{\rho}{4\sqrt n}\left(P-\frac{\tr P}{n} I\right),
 \qquad \tr B=1,\qquad \|B-\widehat A(T)\|\le\frac\rho4,
\end{equation}
where the last estimate follows from $\|P-(\tr P/n)I\|^2=\tr P(1-\tr P/n)\le n$.

Suppose that $\|\widehat A(t)-\widehat A(T)\|\le\rho$, and let $\theta\in\partial B_1\cap\Lambda_t$.
By \eqref{sest123}, we have
$$
 |\widehat A(T)\theta|
 \le\|\widehat A(T)-\widehat A(t)\|
    +|\widehat A(t)\theta|
 \le\rho+C_n\eps_T
 \le2\rho.
$$
Let $e_1,\ldots,e_n$ be an orthonormal eigenbasis of $\widehat A(T)$, with corresponding eigenvalues $\lambda_1,\ldots,\lambda_n$. We can represent $\theta$ as $\theta=\sum_i a_i e_i$, where $\sum_i a_i^2=1$. By the choice of $\rho$, each eigenvalue satisfies either $|\lambda_i|<\rho/K$ or $|\lambda_i|\ge K\rho$. For the latter, we have $|\lambda_i|\le\lambda_i^2/(K\rho)$. Since $|\widehat A(T)\theta|\le2\rho$, we have
\begin{equation}\label{vspec12}
 \begin{split}
 |(I-P)\theta|
 &=\left(\sum_{|\lambda_i|\ge K\rho}a_i^2\right)^{1/2}\\
 &\le\frac1{K\rho}
 \left(\sum_{|\lambda_i|\ge K\rho}\lambda_i^2a_i^2\right)^{1/2}
 \le\frac{|\widehat A(T)\theta|}{K\rho}
 \le\frac2K,\\
 |\theta\cdot\widehat A(T)\theta|
 &\le\sum_i|\lambda_i|a_i^2\\
 &\le\frac\rho K\sum_{|\lambda_i|<\rho/K}a_i^2
 +\frac1{K\rho}\sum_{|\lambda_i|\ge K\rho}\lambda_i^2a_i^2\\
 &\le\frac\rho K+\frac{|\widehat A(T)\theta|^2}{K\rho}
 \le\frac{5\rho}K.
 \end{split}
\end{equation}
Combining \eqref{cmatrix12}, \eqref{vspec12}, and \eqref{sest123}, and using $\tr P\le n-1$, we obtain
$$
 \begin{aligned}
 \theta\cdot(B-\widehat A(t))\theta
 &=\theta\cdot(B-\widehat A(T))\theta
   +\theta\cdot\widehat A(T)\theta
   -\theta\cdot\widehat A(t)\theta\\
 &=\frac{\rho}{4\sqrt n}
   \left(|P\theta|^2-\frac{\tr P}{n}\right)
   +\theta\cdot\widehat A(T)\theta
   -\theta\cdot\widehat A(t)\theta\\
 &=\frac{\rho}{4\sqrt n}
   \left(1-|(I-P)\theta|^2-\frac{\tr P}{n}\right)
   +\theta\cdot\widehat A(T)\theta
   -\theta\cdot\widehat A(t)\theta\\
 &\ge-\frac{5\rho}K
   +\frac{\rho}{4\sqrt n}
    \left(\frac1n-\frac4{K^2}\right)
   -C_n\eps_T\\
 &\ge c_n\rho.
 \end{aligned}
$$
Here we used $|P\theta|^2=1-|(I-P)\theta|^2$, $\tr P\le n-1$, and $C_n\eps_T\le\rho/K^2$, and chose $K$ sufficiently large. Substituting into \eqref{comparison} yields
\begin{equation}\label{cmpode12}
 D_B'(t)\le-c_n\rho\,\tr H(t)
 \qquad\text{while }\|\widehat A(t)-\widehat A(T)\|\le\rho.
\end{equation}
Suppose there exists a first time $t_*$ such that $\|\widehat A(t_*)-\widehat A(T)\|=\rho$. By \eqref{cmpode12}, $D_B$ is nonincreasing on $[T,t_*]$.
Then, by \eqref{compdist} and \eqref{cmatrix12} we have $\|\widehat A(t_*)-B\|\le\|\widehat A(T)-B\|\le\rho/4$. It implies that $\|\widehat A(t_*)-\widehat A(T)\|\le\rho/2$, a contradiction. Hence \eqref{cmpode12} holds for all $t\ge T$. By \eqref{compdist} and \eqref{cmatrix12}, we have that $D_B(T)\leq C_n^2\rho^2.$ Now, since $D_B\ge0$, by integrating \eqref{cmpode12}, we obtain
\begin{equation}\label{Hest123}
 \int_T^\infty\tr H(t)\,dt
 \le\frac{D_B(T)}{c_n\rho}
 \le C_n\rho\le C_n\eps_T.
\end{equation}
In the original variables, the integral in \eqref{Hest123} equals
$\int_{\Lambda(u)\cap B_{e^{-T}}(x)}|z-x|^{-n}\,dz$.

Since $F$ is bounded, by \eqref{mevolu111}, we have
\begin{equation}\label{Fintest111}
 F(t)=\int_t^\infty e^{-(n+2)(s-t)}\tr H(s)\,ds.
\end{equation}
By Fubini, we have
\begin{equation}\label{varbounds}
 \int_T^\infty F(t)\,dt
 \le\frac1{n+2}\int_T^\infty\tr H(s)\,ds,
 \end{equation}
By \eqref{projeq111} and $\|M\|\le F$, we have
 \begin{equation}\label{intbound2}
 \int_T^\infty\|A'(t)\|\,dt
 \le C_n\int_T^\infty F(t)\,dt.
\end{equation}
Combining \eqref{varbounds}, \eqref{intbound2} and \eqref{Hest123},  we have the desired \eqref{variation}.
\end{proof}

\begin{proof}[Proof of the continuity in Theorem~\ref{continuity}] Let $K\Subset D$ be a compact subset, and let $A_z(t)$ and $E_z(t)$ be the projection matrix and its error at $z\in K\cap\Sigma(u)$. Denote
$$
 \omega_K(t)=\sup_{\substack{z\in K\cap\Sigma(u)\\ s\ge t}}E_z(s).
$$
By \eqref{projbound}, we have $\omega_K(t)\to0$ as $t\to\infty$. For all sufficiently large $t$,  by \eqref{variation} and the convergence $A_z(s)\to A_z$, we have
\begin{equation}\label{unifconti}
 \|A_z(t)-A_z\|
 \le\int_t^\infty\|A_z'(s)\|\,ds
 \le C_n\sup_{s\ge t}E_z(s)
 \le C_n\omega_K(t).
\end{equation}
Recall from \eqref{projerror} that, for $r=e^{-t}$,
$$
 \|Du_{z,r}-A_z(t)y\|_{L^\infty(B_1)}=E_z(t).
$$
Hence
$$
 \begin{aligned}
 \|Du_{z,r}-A_zy\|_{L^\infty(B_1)}
 &=\|Du_{z,r}-A_z(t)y+(A_z(t)-A_z)y\|_{L^\infty(B_1)}\\
 &\le E_z(t)+\|A_z(t)-A_z\|.
 \end{aligned}
$$
Since $u_{z,r}(0)=q_z(0)=0$, by integrating along line segments, together with the preceding estimate and \eqref{unifconti}, we have
\begin{equation}\label{uniferror}
 \|u_{z,r}-q_z\|_{C^1(\overline B_1)}
 \le C_n\sup_{s\ge t}E_z(s).
\end{equation}
The gradients of $u_{z,r}$ and $q_z$ have uniformly bounded Lipschitz seminorms by the local Hessian estimates. Thus interpolation and \eqref{uniferror} imply, for every $0<\alpha<1$,
\begin{equation}\label{unifconv}
 \lim_{r\downarrow0}\sup_{z\in K\cap\Sigma(u)}
 \|u_{z,r}-q_z\|_{C^{1,\alpha}(\overline B_1)}=0.
\end{equation}

Now let $x_j\to x$, with $x_j,x\in\Sigma(u)$, and choose $K\Subset D$ containing a neighborhood of $x$. For each fixed sufficiently large $T$, the rescalings $u_{x_j,e^{-T}}$ converge to $u_{x,e^{-T}}$ in $C^1(\overline B_1)$. By the continuity of the projection, $A_{x_j}(T)\to A_x(T)$. Applying \eqref{unifconti} at $x_j$ and $x$, we obtain
$$
 \|A_{x_j}-A_x\|
 \le\|A_{x_j}(T)-A_x(T)\|+2C_n\omega_K(T).
$$
By letting $j\to\infty$ and then $T\to\infty$  we have the desired continuity.
\end{proof}

Rescaling \eqref{uniferror}, we also obtain
\begin{equation}\label{expansions}
 \begin{split}
 u(x+h)&=u(x)+q_x(h)+o(|h|^2),\\
 Du(x+h)&=A_xh+o(|h|)
 \end{split}
\end{equation}
uniformly for $x$ in a compact subset of $\Sigma(u)$.

\begin{proof}[Proof of Theorem~\ref{manifolds}]
The regular set is relatively open in $\Gamma(u)$ by Lemma~\ref{halfspace}, so $\Sigma(u)$ is relatively closed in $D$. By Theorem~\ref{continuity} and \eqref{expansions}, the jets $u(x)+q_x(\,\cdot-x)$ satisfy the Whitney compatibility conditions. The conclusion then follows from the Whitney extension and implicit function argument in the proof of \cite[Theorem~7.9]{PSU2012}.
\end{proof}

\begin{remark}\label{nosignrem}
Theorem~\ref{continuity} also applies to the no-sign obstacle problem
\begin{equation}\label{nosigneq}
 \Delta u=1-\chi_{\Lambda_0(u)}
 \quad\text{in }D,
 \qquad
 \Lambda_0(u)=\{u=0,\ Du=0\}.
\end{equation}
Let $u\in C^{1,1}_{\mathrm{loc}}(D)$ solve \eqref{nosigneq}. Define $\Gamma_0(u)=\partial\Lambda_0(u)\cap D$, and let $\Sigma_0(u)$ consist of those $z\in\Gamma_0(u)$ at which some blow-up $r^{-2}u(z+r\,\cdot)$ is a homogeneous quadratic polynomial with Laplacian one.

Since $Du$ is locally Lipschitz, \cite[Lemma~7.7]{GilbargTrudinger2001} implies that $D^2u=0$ almost everywhere on $\Lambda(u)=\{Du=0\}$. On $\Lambda(u)\setminus\Lambda_0(u)$, however, \eqref{nosigneq} implies that $\Delta u=1$ almost everywhere. Hence $|\Lambda(u)\setminus\Lambda_0(u)|=0$.
Since $\Lambda_0(u)\subset\Lambda(u)$, their characteristic functions agree almost everywhere, and $u$ solves \eqref{maineq}. It is straightforward to verify that
\begin{equation}\label{nsboundary}
 \Gamma_0(u)=\Gamma(u)\cap\{u=0\}
\end{equation}
and
\begin{equation}\label{nssingular}
 \Sigma_0(u)=\Sigma(u)\cap\{u=0\}.
\end{equation}
By Theorem~\ref{continuity} and \eqref{nssingular}, at each $z\in\Sigma_0(u)$ the quadratic blow-up
$$
 q_z(y)=\tfrac12 y\cdot A_zy
$$
is unique, and $z\mapsto A_z$ is continuous on $\Sigma_0(u)$. The uniqueness was proved by Chen, Feng and Li
\cite[Theorem~1.1]{ChenFengLi2022}.
\end{remark}

\section*{Acknowledgement}
The AI tool ChatGPT (OpenAI) was used in the preparation of this manuscript. The authors have reviewed all AI-assisted content and take full responsibility for the final manuscript.

\end{document}